\documentclass[final,3p,nopreprintline]{elsarticle}

\usepackage[utf8]{inputenc}
\usepackage[english]{babel}
\usepackage{amsmath,amssymb,amsthm,mathtools,array}
\usepackage{upgreek}
\usepackage{enumitem}
\usepackage{hyperref} 
\usepackage[capitalize]{cleveref} 
\usepackage{ifthen}

\hypersetup{colorlinks=true}

\newtheorem{lemma}{Lemma}
\crefname{lemma}{Lemma}{Lemmas}
\newtheorem{proposition}[lemma]{Proposition}
\crefname{proposition}{Proposition}{Propositions}
\newtheorem{remark}{Remark}
\crefname{remark}{Remark}{Remarks}

\newcommand*{\N}{\mathbb{N}}
\newcommand*{\R}{\mathbb{R}}

\newcommand*{\Hilbert}{\mathcal{H}}
\newcommand*{\Manifold}{\mathcal{M}_\Phi}
\newcommand*{\ParameterDomain}{\mathcal{Z}}
\newcommand*{\TangentSpace}[1][]{\mathcal{T}\ifthenelse{\equal{#1}{}}{}{_{#1}}\Manifold}
\newcommand*{\ConfigurationDomain}{\mathcal{D}}
\newcommand*{\Domain}{\Omega}
\newcommand*{\dm}{D}
\renewcommand{\vec}[1]{\boldsymbol{#1}}
\newcommand*{\tensor}[1]{\boldsymbol{\mathbf{#1}}}
\newcommand*{\Id}{\tensor{I}}
\newcommand*{\Flow}{\tensor{\Phi}}
\providecommand*{\tr}{\operatorname{tr}}
\newcommand*{\Transpose}{^\top}
\newcommand*{\PDiff}[1]{\partial_{#1}}
\newcommand*{\Diff}[1][]{\mathrm{d}#1}
\newcommand*{\Grad}[1][]{\nabla#1}
\newcommand*{\Div}[1][]{\nabla#1\cdot}
\newcommand*{\Laplace}[1][]{\Delta#1}
\newcommand*{\GradX}{\Grad[_{\!\Position}]}
\newcommand*{\GradHatQ}{\Grad[_{\!\hat\Configuration}]}
\newcommand*{\GradQ}{\Grad[_{\!\Configuration}]}
\newcommand*{\DivX}{\Div[_{\!\Position}]}
\newcommand*{\DivHatQ}{\Div[_{\!\hat\Configuration}]}
\newcommand*{\DivQ}{\Div[_{\!\Configuration}]}
\providecommand\given{}
\newcommand\SetSymbol[1][]{\nonscript\:#1\vert\allowbreak\nonscript\:\mathopen{}}
\DeclarePairedDelimiterX{\Set}[1]{\{}{\}}{\renewcommand{\given}{\SetSymbol[\delimsize]}#1}
\newcommand*{\Matsym}[1][\dm]{\R^{#1\times #1}_{\mathrm sym}}
\newcommand*{\Configuration}{\vec{q}}
\newcommand*{\ExtraStress}{\tensor{\uptau}}
\newcommand*{\FPpdf}{\psi}
\newcommand*{\Position}{\vec{x}}
\newcommand*{\Pressure}{p}
\newcommand*{\RouseMatrix}{\tensor{R}}
\newcommand*{\Strain}{\tensor{\upvarepsilon}}
\newcommand*{\Stress}{\tensor{\upsigma}}
\newcommand*{\Time}{t}
\newcommand*{\Velocity}{\vec{u}}
\newcommand*{\Parameter}{z}
\newcommand*{\CenterDiffusion}{\varepsilon}
\newcommand*{\Deborah}{\mathrm{De}}
\newcommand*{\Force}{\vec{F}}
\newcommand*{\Reynolds}{\mathrm{Re}}

\newcommand*{\Conformation}{\tensor{C}}
\newcommand*{\Covariance}{\tensor{\mathcal{C}}}
\newcommand*{\D}{\mathcal{D}_A}

\begin{document}

\begin{frontmatter}


\title{An equivalence of moment closure and nonlinear variational approximation of the Fokker--Planck equation for dilute polymeric flow}

\author{Caroline Lasser}
\ead{classer@tum.de}
\ead[url]{https://orcid.org/0000-0001-7272-2510}

\author{Stephan B. Lunowa\corref{cor1}}
\ead[url]{https://orcid.org/0000-0002-5214-7245}
\ead{stephan.lunowa@tum.de}
\cortext[cor1]{Corresponding author}

\author{Barbara Wohlmuth}
\ead{wohlmuth@tum.de}
\ead[url]{https://orcid.org/0000-0001-6908-6015}

\affiliation{organization={Technical University of Munich, School of Computation, Information and Technology, Department of Mathematics},
            addressline={Boltzmannstraße~3}, 
            city={D-85748 Garching},
            country={Germany}}

\begin{abstract}
We establish the equivalence between a classical moment closure and a nonlinear variational approximation of the Fokker--Planck equation for dilute polymeric flow in the linearized Hookean spring chain setting. The variational formulation is based on the Dirac--Frankel principle applied to a Gaussian approximation manifold endowed with the Fisher--Rao information metric. We show that the invariance of this manifold under the linear configurational dynamics yields an exact evolution for the macroscopic conformation tensor, recovering the classical diffusive Oldroyd-B closure. While the equivalence only holds in the linearized setting, the associated variational framework provides an abstract error representation. Thus it can serve in future work as a starting point for the systematic construction of reduced approximation schemes for polymeric flows with nonlinear forcing laws.
\end{abstract}



\end{frontmatter}

\section{Introduction}
\label{sec:intro}
The dynamics of dilute polymeric fluids are modeled through a coupled system of partial differential equations, where the macroscopic velocity field is linked to the microscopic configuration of polymer chains via a Fokker--Planck equation \cite{Bird1987}.
This equation describes the evolution of a probability density in configuration space, driven by deterministic drift terms representing polymer stretching and relaxation (induced by the flow and internal elasticity) as well as diffusion terms modeling thermal fluctuations.
The Fokker--Planck equation is coupled to the Navier--Stokes equations through the polymeric extra stress tensor, which is expressed as a moment of the probability density.
This coupling gives rise to a nonlinear feedback resulting in visco-elastic behavior.

Besides various analytical challenges, the direct numerical approximation is highly demanding as the Fokker--Planck equation for a linear chain of $N$ segments is formulated over the Cartesian product of $N+1$ domains of dimension $d$ (the dimension of physical space).
Moreover, in many applications, we are not interested in the detailed behavior of the microscopic probability density, but 
in the macroscopic quantities of the flow.
Hence, deriving macroscopic visco-elastic models as asymptotic limits of the kinetic description has been of big interest throughout the last decades.
Such closure relations have been derived mostly using moment methods combined with either ad-hoc assumptions on the structure or using quasi-equilibrium conditions, see, e.g., \cite{Gorban2001,Ilg2002,Hyon2008} and the references therein.
Only in the case of the  Hookean  linear chain model for polymer molecules, an exact closure model for the extra stress can be derived, resulting in the (diffusive) Oldroyd-B model \cite{Bird1987,Debiec2025}.
This closure is typically obtained by taking second moments of the Fokker--Planck equation or, alternatively, by a Hermite spectral approximation which results in a small lower triangular matrix \cite{Hetland2023,Beddrich2024,Beddrich2024DragReduction}.

In this article, we discuss an alternative approach using nonlinear variational approximation in a Gaussian manifold.
In contrast to linear variational approximations, such as, e.g., classical finite elements, one seeks for a nonlinear parameterization that improves the accuracy while preserving the underlying variational structure.
Here, this approach is based on the observation that Fokker--Planck equations are gradient flows with respect to the Wasserstein metric on probability density functions \cite{Jodan1998}, as also discussed in \cite{zhang2024Nonlinpar} regarding different metrics.
In particular, the Fisher--Rao information metric is diffeomorphism invariant and thus the most natural choice, cf. also
\cite{Anderson2024,Bruna2024NeuralGalerkin,Chen2024}, where similar parametrization are discussed in the context of reduced-order approximation of stochastic problems and in the context of neural-network training after discretization in time.

Combining these ideas, we show in this article that the moment closure and the nonlinear variational approach are equivalent in the sense that they result in the same macroscopic approximation for the Hookean chain model and give rise to an exact representation of the Navier--Stokes system.
Moreover, the variational approach allows us to derive an error estimate for the probability distribution.
To establish these results, the article is organized as follows.
In the following \cref{sec:model}, we introduce the Fokker--Planck equation for dilute polymeric flow and discuss the macroscopic approximation obtained by the moment closure.
In \cref{sec:ansatz}, we introduce the nonlinear variational approach, and discuss the specific choice of a Gaussian manifold in \cref{sec:hookean}, in particular, focusing on the resulting macroscopic closure and the error representation.
Finally, \cref{sec:conclusion} closes with a summary and an outlook to dilute polymer models with nonlinear forces.

\section{Governing equations}
\label{sec:model}
We denote the time with $t \in (0, T)$, the position in the Lipschitz-bounded macroscopic flow domain with $\Position \in \Domain \subset \R^d$, $d = 2,3$.
The velocity $\Velocity(\Time, \Position)$ and pressure $\Pressure(\Time, \Position)$ of the solvent fluid are governed by the incompressible Navier--Stokes equations
\begin{align}
    \label{eq:NS}
    \PDiff{\Time} \Velocity + \left( \Velocity \cdot \GradX \right) \Velocity &= \DivX \Stress, \qquad
    \DivX \Velocity = 0 \qquad
    \text{in } (0,T) \times \Domain,
\end{align}
where the stress $\Stress$ consists of the Newtonian part and an extra-stress $\ExtraStress$, viz.
\begin{align}
    \label{eq:NS:stress}
    \Stress &= \frac{\beta}{\Reynolds} \Strain(\Velocity) - \Pressure \Id_d + \frac{1 - \beta}{\Reynolds\,\Deborah} \ExtraStress, \qquad
    \Strain(\Velocity) = \frac{1}{2} \left( \GradX\Velocity + \left(\GradX\Velocity\right)\Transpose \right)
\end{align}
Here, $\Reynolds$ denotes the Reynolds number, $\Deborah$ the Deborah number and $\beta = \eta_s / (\eta_s + \eta_p)$ the viscosity ratio of the solvent viscosity $\eta_s$ and the zero-shear-rate polymeric viscosity $\eta_p$. Finally, $\Id_d$ is the identity matrix on $\R^d$.
For simplicity, we consider homogeneous Dirichlet boundary conditions $\Velocity |_{(0,T) \times \partial\Domain} = \vec{0}$ and a smooth, divergence-free initial condition $\Velocity|_{\Time = 0} = \Velocity_0$, which satisfies the boundary condition.

On the microscopic scale, we model a single polymer molecule as a freely jointed linear bead-spring chain, the so-called Rouse model \cite[Chp.~15]{Bird1987}.
This chain consists of $N \in \N$ (identical) Hookean springs which connect $N+1$ identical mass-less beads at positions $\vec{r}_n \in \R^d$, $n=0,\dots,N$.
We identify the center of each molecule with the spatial coordinate
\[ \Position = \frac{1}{N+1}\sum_{n=0}^N \vec{r}_n \]
and denote the distance vectors $\hat{\Configuration}_n = \vec{r}_n - \vec{r}_{n-1}$, $n=1,\dots,N$, while the full vector is denoted $\hat{\Configuration} = (\hat\Configuration_1, \dots, \hat\Configuration_N)\Transpose \in \ConfigurationDomain = \R^{\dm}$ with $\dm = Nd$.
Furthermore, the probability density function (PDF) $\hat\FPpdf(\Time, \Position, \hat\Configuration)$ expresses the probability that at time $\Time$, there is a polymer molecule with center-of-position $\Position$ and configuration vector $\hat\Configuration$.

We assume that the polymer dynamics in terms of the PDF $\hat\FPpdf$ are described by the following Fokker--Planck (FP) equation (cf. \cite[eq.~(15.1-7)]{Bird1987})
\begin{equation}\label{eq:FP}
\PDiff{\Time} \hat\FPpdf + \mathcal{L}_{\Position} \hat\FPpdf
    + \mathcal{L}_{\hat\Configuration}\hat\FPpdf  = 0 ,
    \qquad \text{in } (0,T) \times \Domain \times \ConfigurationDomain, 
\end{equation}
with spatial and configurational differential operators
\begin{align*}
    \mathcal{L}_{\Position}\hat\FPpdf &:= \Velocity \cdot \GradX \hat\FPpdf - \CenterDiffusion \Laplace_{\Position} \hat\FPpdf, \\
    \mathcal{L}_{\hat\Configuration}\hat\FPpdf &:= \DivHatQ \Big((\Id_N \otimes \GradX\Velocity) \hat\FPpdf \hat\Configuration - \frac{1}{2\Deborah} (\RouseMatrix \otimes \Id_d) \big(\hat\FPpdf \Force(\hat\Configuration) + \GradHatQ \hat\FPpdf\big) \Big).
\end{align*}
Here $\CenterDiffusion>0$ denotes the center-position diffusion coefficient, and the entropic spring force $\Force(\hat\Configuration)$ is in our case given by a normalized linear Hookean model $\Force(\hat\Configuration) = \hat\Configuration$. For the special case of a oriented linear chain,  the symmetric and positive definite Rouse matrix is given by $\RouseMatrix = \mathrm{tridiag}(-1, 2, -1) \in \R^{N \times N}$, see \cite[Chp.~15]{Bird1987}.

\begin{remark}
    In more general chain configurations, the chain can be represented as a connected graph where the number of edges is possibly much larger than the number of vertices and different properties can be attributed to the beads, see, e.g., block co-polymers.
    Then, the Fokker--Planck equation must be based on the positions of the beads, so that the $\RouseMatrix$ reflects the associated graph Laplacian, while the force term $\Force$ involves the incidence matrix of beads connected by a spring.
    Moreover, it is possible to take hydrodynamic self-interaction into account. The first order approximation leads to the Zimm model, in which the Rouse matrix is augmented by nonlinear drag terms, see \cite{Zimm56}.
\end{remark}

\begin{remark}
    Alternatively to the Hookean force, finitely extensible nonlinear elastic (FENE) models exist, which only allows for finite distance vectors $\hat\Configuration_n$ within an open ball, while the spring force tends to infinity as the boundary of the configuration space is approached, i.e.,
    \begin{equation*}
        \Force_n(\hat\Configuration_n) = \frac{\hat\Configuration_n}{1 - | \hat\Configuration_n|^2 / q_{\max}^2}, \qquad\hat\Configuration_n \in \Set*{ \vec{v} \in \R^d \given |\vec{v}| < q_{\max}},
    \end{equation*}
    where $q_{\max} > 0$ denotes the maximal distance between two consecutive beads and $\Force_n$ is the vector of force components with respect to the distance vector $\hat\Configuration_n$.
    As $q_{\max}$ tends to infinity, one recovers the linear Hookean spring force \cite{Warner1972}.
    More elaborate FENE models exist, see e.g. \cite{Jedynak2015}.
\end{remark}

To ensure mass conservation, we consider homogeneous Neumann boundary conditions for \cref{eq:FP} on the boundary of the macroscopic domain $\Omega$, i.e.,
$\GradX\hat\FPpdf \cdot \vec{n} |_{(0,T) \times \partial\Domain \times\ConfigurationDomain} = 0$  with $\vec{n}$ being the outer unit normal vector to $\partial\Domain$, and a smooth, nonnegative initial condition $\hat\FPpdf|_{\Time = 0} = \hat\FPpdf_0$, which satisfies the boundary condition, symmetry in $\Configuration$, i.e., $\hat\FPpdf_0(\Position,\hat\Configuration) = \hat\FPpdf_0(\Position,-\hat\Configuration)$, and $\int_{\ConfigurationDomain} \hat\FPpdf_0(\Position,\hat\Configuration) \,\Diff[\hat\Configuration] = 1$ everywhere in $\Domain$.

Finally, the probability density $\hat\FPpdf$ enters in the extra-stress tensor $\ExtraStress\in\Matsym[d]$ in the NS equations, in the form of Kramers' expression
\begin{align}
    \label{eq:extraStress}
    \ExtraStress(\Time,\Position) = \int_\ConfigurationDomain \hat\FPpdf(\Time, \Position, \hat\Configuration) \left( \sum_{n=1}^N \Force_n(\hat\Configuration_n) \hat\Configuration_n \Transpose  - N \Id_d \right) \,\Diff[\hat\Configuration],
\end{align}
where $\Force_n$ denotes the force components with respect to the distance vectors~$\hat\Configuration_n$.
Since only the divergence of the extra-stress tensor enters the Navier--Stokes equation, the conformation tensor $\Conformation \in \Matsym[d]$ is often used instead
\begin{align*}
    \Conformation(\Time,\Position) = \int_\ConfigurationDomain \hat\FPpdf(\Time,\Position,\hat\Configuration) \sum_{n=1}^N \Force_n(\hat\Configuration_n) \hat\Configuration_n \Transpose \,\Diff[\hat\Configuration] = \ExtraStress(\Time,\Position) + N \Id_d .
\end{align*}
Note that under these conditions, the Fokker--Planck \cref{eq:FP} has a well developed well-posedness theory in Maxwellian-weighted Sobolev spaces 
\begin{align*}
&\Hilbert = \left\{ \phi \in L^1_{\mathrm{loc}}(\Omega\times\ConfigurationDomain): \|\phi\|_{\Hilbert}<\infty\right\},\\
&\|\phi\|_\Hilbert^2 =  \int_{\Omega\times\ConfigurationDomain} \mu \left(|\phi|^2 + |\nabla_{\Position}\phi|^2 + |\nabla_{\Configuration}\phi|^2\right) \Diff[\Position] \Diff[\Configuration] ,
\end{align*} 
where $\mu$ denotes the Maxwellian, 
see e.g. \cite{Barrett2018,Debiec2025}.

\subsection{The Hookean linear-chain model}

For a linear Hookean force $\vec{F}(\hat\Configuration) = \hat\Configuration$ with configuration $\hat\Configuration \in \ConfigurationDomain = \R^{\dm}$, it is convenient to orthogonalize the Fokker--Planck equation with respect to the Rouse matrix $\RouseMatrix$. 
Changing the configuration coordinates to $\Configuration := (\tensor{Q} \otimes \Id_d)\hat\Configuration$, where $\tensor{Q}$ is the matrix of eigenvectors of the decomposition $\RouseMatrix = \tensor{Q}\Transpose\tensor{\Lambda}\tensor{Q}$, the configurational operator for the PDF $\FPpdf(\Time,\Position,\Configuration) := \hat\FPpdf(\Time,\Position,\hat\Configuration)$ takes the form
\begin{align*}
    \mathcal{L}_{\Configuration}\FPpdf &= \DivQ \Big[ (\tensor{Q}\otimes \Id_d) \Big((\Id_N \otimes \GradX\Velocity) (\tensor{Q}\Transpose \otimes \Id_d) \FPpdf \Configuration \\&\qquad\qquad
        - \frac{1}{2\Deborah} ((\tensor{Q}\Transpose\tensor{\Lambda} \tensor{Q}) \otimes \Id_d) \big((\tensor{Q}\Transpose \otimes \Id_d) \FPpdf\Configuration + (\tensor{Q}\Transpose \otimes \Id_d) \GradQ \FPpdf\big) \Big) \Big] \\
     &= \DivQ \Big((\Id_N \otimes \GradX\Velocity) \FPpdf \Configuration - \frac{1}{2\Deborah} (\tensor{\Lambda} \otimes \Id_d) \big(\FPpdf \Configuration + \GradQ \FPpdf\big) \Big) ,
\end{align*}
while the spatial operator remains the same.
Moreover, the conformation and extra-stress tensors then become
\[
    \Conformation(\Time,\Position) = \int_\ConfigurationDomain \FPpdf(\Time,\Position,\Configuration) \sum_{n=1}^N \Configuration_n \Configuration_n \Transpose \,\Diff[\Configuration] = \ExtraStress(\Time,\Position) + N \Id_d .
\]

Introducing the time and position dependent tensor $\tensor{M}(\Time,\Position)\in\R^{\dm\times\dm}$, 
\[
    \tensor{M}(\Time,\Position) = \frac{1}{2\Deborah}\tensor{\Lambda}\otimes\Id_d - \Id_N \otimes \GradX\Velocity(\Time,\Position),
\]
the configurational operator can be written as
\begin{equation}\label{eq:Lc_Hook}
    \mathcal{L}_{\Configuration} = -\DivQ \Big( \tensor{M}\Configuration + \frac{1}{2\Deborah} (\tensor{\Lambda}\otimes\Id_d)\GradQ\Big) .
\end{equation}
Note that all involved tensors are block-diagonal with $N$ blocks of size $d \times d$.
This will lead to a decoupling into $N$ subproblems of size $d \times d$ instead of a fully coupled system of 
size $\dm\times\dm = Nd \times Nd$.

If the flow is homogeneous in the sense that the Jacobian of the velocity is constant in time and symmetric, 
$\GradX\Velocity = \GradX\Velocity\Transpose$, and if $\GradX\Velocity$ is ``small enough'', such that $\tensor{M}$ 
is positive-definite, then the solution of the steady-state FP equation $\mathcal{L}_{\Configuration}\FPpdf = 0$ is given by the Gaussian (cf. \cite[eq. (15.1-8)]{Bird1987})
\begin{equation*}
\FPpdf(\Position,\Configuration) = \frac{\exp\left(- \Deborah\, \Configuration\Transpose (\tensor{\Lambda}\otimes\Id_d)^{-1}\tensor{M}(\Position) \Configuration \right)}{\int_\ConfigurationDomain \exp(- \Deborah\, \Configuration\Transpose (\tensor{\Lambda}\otimes\Id_d)^{-1}\tensor{M}(\Position) \Configuration) \,\Diff[\Configuration]} .
\end{equation*}
Indeed, for a symmetric, positive definite matrix $\tensor{B}$, we have
\[
\GradQ \exp(-\Configuration\Transpose \tensor{B}\Configuration) = -2 \tensor B\Configuration \exp(-\Configuration\Transpose \tensor{B}\Configuration) 
\]
Since $\tensor{M}$ and $\tensor{\Lambda}\otimes\Id_d$ commute, $\tensor{B} = \Deborah\,(\tensor{\Lambda}\otimes\Id_d)^{-1}\tensor{M} = \tensor{B}\Transpose$ is symmetric and positive-definite. We thus have
\[
\GradQ \FPpdf(\Position,\Configuration) = -2\Deborah\,(\tensor{\Lambda}\otimes\Id_d)^{-1}\tensor{M}(\Position)\Configuration\,\FPpdf(\Position,\Configuration)
\]
and
\[
\tensor{M}(\Position)\Configuration\,\FPpdf(\Position,\Configuration) + 
\frac{1}{2\Deborah}(\tensor{\Lambda}\otimes\Id_d)\GradQ \FPpdf(\Position,\Configuration) 
 = 0. 
 \]
As the matrix $(\tensor{\Lambda}\otimes\Id_d)^{-1}\tensor{M}(\Position)$ is block-diagonal, the steady-state can be equivalently stated as
\begin{equation}\label{eq:stationary}
    \FPpdf(\Position,\Configuration) = \prod_{n=1}^N \frac{\exp\left(- \Deborah\,\lambda_n^{-1} \Configuration_n\Transpose \tensor{M}_n(\Position) \Configuration_n \right)}{\int_{\R^d} \exp(- \Deborah\,\lambda_n^{-1} \Configuration_n\Transpose \tensor{M}_n(\Position) \Configuration_n) \,\Diff[\Configuration_i]} ,
\end{equation}
where $\lambda_n$ are the diagonal elements of $\tensor{\Lambda}$, which are also the eigenvalues of the Rouse matrix $\tensor{R}$, and $\tensor{M}_n \in \mathbb{R}^{d \times d}$ is defined as the $n$-th diagonal block of $\tensor{M}$, i.e., by 
\begin{align*}
  \tensor{M}_n =  \frac{\lambda_n}{2 \Deborah} \Id_d - \GradX u(\Time,\Position).
\end{align*}

More generally, when formally testing the full Fokker--Planck equation \eqref{eq:FP} with $\Configuration\Configuration\Transpose$ \cite[Sec.~15.3]{Bird1987}\cite{Debiec2025} or equivalently by a Hermite spectral approximation up to second degree \cite{Hetland2023,Beddrich2024,Beddrich2024DragReduction}, one obtains an exact macroscopic closure relation for the extra-stress tensor $\ExtraStress = \ExtraStress[\FPpdf](\Time,\Position)$ of the Fokker-Planck solution $\FPpdf$.

Using the upper convected derivative, given for a tensor $\tensor{\kappa}$ by
\begin{equation*}
    \frac{\mathcal{D}}{\mathcal{D}\Time} \tensor{\kappa} = \PDiff{\Time} \tensor{\kappa} + \left(\Velocity \cdot \GradX \right) \tensor{\kappa}
        - \left( \GradX\Velocity\, \tensor{\kappa}  + \tensor{\kappa}\, \GradX\Velocity\Transpose \right),
\end{equation*}
and the block-diagonal structure, the resulting set of partial differential equations reads
\begin{align*}
    \frac{\mathcal{D}}{\mathcal{D}\Time} \Conformation_n
    &= \CenterDiffusion\Laplace_{\Position} \Conformation_n - \frac{\lambda_n}{\Deborah} (\Conformation_n - \Id_d) ,\qquad\text{in }(0,T) \times \Omega, ~ n=1,\dots,N,
\end{align*}
where the conformation tensor is given by $\Conformation = \sum_{n=1}^N \Conformation_n$.
These equations can be also stated in the form
\begin{align}
    \label{eq:conformation}
    \left(\PDiff{\Time} + \Velocity \cdot \GradX - \CenterDiffusion\Laplace_{\Position}\right) \Conformation_n
    &= \frac{\lambda_n}{\Deborah}\Id_d - \tensor{M}_n \Conformation_n - \Conformation_n \tensor{M}_n\Transpose \qquad\text{in }(0,T) \times \Omega .
\end{align}
Hence, the extra-stress tensor is composed by $\ExtraStress = \sum_{n=1}^N \ExtraStress_n$ with the tensors $\ExtraStress_n := \Conformation_n - \Id_d$ satisfying
\begin{align} \label{eq:extrai}
    \frac{\mathcal{D}}{\mathcal{D}\Time} \ExtraStress_n
        &= \CenterDiffusion\Laplace_{\Position} \ExtraStress_n
        - \frac{\lambda_n}{\Deborah} \ExtraStress_n + \GradX\Velocity + \GradX\Velocity\Transpose, \qquad\text{in }(0,T) \times \Omega, ~ n=1,\dots,N.
\end{align}
Note that this is the diffusive Oldroyd-B model, see \cite{Debiec2025} for further details, in particular for a specification of the function spaces and the assumptions on initial data, which guarantee well-posedness. We note that it is precisely equation~\eqref{eq:extrai}, which we will re-derive here when applying the Dirac--Frenkel variational principle to a Gaussian manifold in \Cref{sec:hookean}.

According to the microscopic model, the macroscopic model is closed by homogeneous Neumann boundary conditions $\GradX\Conformation_n\cdot \vec{n} |_{(0,T) \times \partial\Domain} = \tensor{0}$, and the smooth initial condition $\Conformation_n |_{\Time = 0}  := \int_{\ConfigurationDomain} \FPpdf_0(\Configuration) \Configuration_n\Configuration_n\Transpose \,\Diff[\Configuration]$. Here we assume that the initial condition is given in such a way that the symmetric $\Conformation_n $ is positive definite and that $\int_{\ConfigurationDomain} \FPpdf_0(\Configuration) \Configuration_n\Configuration_m\Transpose \,\Diff[\Configuration] = \tensor{0} $ for $n \neq m$.

\section{Nonlinear variational approximation}
\label{sec:ansatz}

Our main result relies on two ingredients: an abstract variational projection principle, and
the invariance of the Gaussian manifold under the Hookean configurational operator.
In this section, we provide the formalism independently of the Gaussian choice,
while in \cref{sec:hookean}, the specific application is presented.

Here, we introduce the abstract framework for our time-depen\-dent variational approximation. The solution is approximated by a parametrized ansatz manifold of probability densities, and its time evolution is determined via a Dirac–Frenkel type variational principle that enforces weighted orthogonality of the residual with respect to the manifold’s tangent space. 

\subsection{Time-dependent variational principle}

Throughout this work, the evolution equation is considered on a Hilbert space $\Hilbert$. The approximation manifold $\Manifold\subset\Hilbert$ is assumed to admit a smooth parametrization $\Phi$. The reduction is defined by the Dirac–Frenkel variational principle and is understood as an evolution equation for the parameter variables. More precisely, we consider in a Hilbert space $\Hilbert$ an abstract evolution problem
\begin{equation}\label{eq:evolution}
\PDiff{\Time}\FPpdf(\Time,\Position,\Configuration) =
\mathcal{L}(\Time)\FPpdf(\Time,\Position,\Configuration),\qquad \FPpdf(0) = \FPpdf_0
\end{equation}
for an initial condition $\FPpdf_0(\Position,\Configuration) \in \Hilbert$, that is for almost every position $\Position\in\Domain\subset\R^d$ a probability density function on an open set $\ConfigurationDomain \subseteq \R^\dm$, and a linear operator~$\mathcal L(\Time)$, that ensures that the solution $\FPpdf(\Time,\Position,\cdot)$ is a probability density function on $\ConfigurationDomain$ for all times $\Time$ and positions $\Position \in \Domain$.  
We assume that the evolution operator is compatible with integration over the configuration variable: there exists a spatial operator 
\(\mathcal L_{\Position}(t)\) such that
\begin{equation}\label{eq:mass}
\int_{\mathcal D}\mathcal L(t)\varphi(x,q)\Diff[\Configuration]
=
-\mathcal L_{\Position}(t)
\left(
\int_{\mathcal D}\varphi(x,q)\Diff[\Configuration]
\right)
\end{equation}
for every sufficiently regular function \(\varphi\). We further assume that
\(\mathcal L_{\Position}(t)1=0\) to ensure propagation of normalization.
We consider time-dependent approximations $f(\Time)\approx\FPpdf(\Time)$ of the form 
\[
    f(\Time,\Position,\Configuration) = \Phi(\Parameter(\Time,\Position),\Configuration)
    \qquad \text{in}\ (0,T)\times\Domain\times\ConfigurationDomain,
\]
that are induced by a smooth, positive, nonlinear parametrization map $\Phi: \ParameterDomain\times\ConfigurationDomain\to(0,\infty)$ based on some open finite-dimensional parameter domain $\ParameterDomain \subset \mathbb{R}^M$. Candidates for such an ansatz are Gaussians, multi-Gaussians or neural networks with position dependent parameters.
We systematically construct $f(t)\in\Manifold$ from the corresponding approximation manifold 
\[
\Manifold = \Set*{ f\in\Hilbert \given f(\Position,\Configuration) = \Phi(\Parameter(\Position),\Configuration)\ \text{in}\ \Domain\times\ConfigurationDomain, \ \Parameter: \Domain\to\ParameterDomain },
\]
using a variant of the Dirac--Frenkel principle, see e.g. \cite{Lubich2005,LasserSu2022}, that is based on a weighted linear least squares approximation that only acts on the configurational degrees of freedom. 
Given the approximation $f(\Time)\in\Manifold$ at time~$\Time$, we require that for a.e. $(\Time,\Position) \in (0,T) \times \Omega$ it holds
\begin{gather}
    \PDiff{\Time} f(\Time) \in \TangentSpace[f(\Time)] \qquad\text{is such that} \notag \\
    \int_{\ConfigurationDomain} \varphi(\Position,\Configuration)\, \left(\PDiff{\Time} - \mathcal{L}(\Time)\right) f(\Time,\Position,\Configuration)\,
    \frac{\Diff[\Configuration]}{f(\Time,\Position, \Configuration)} = 0,\qquad
    \forall\varphi \in \TangentSpace[f(\Time)] \label{eq:var},
\end{gather}
where the tangent space $\TangentSpace[f(\Time)]\subset\Hilbert$ is the linear space that contains the tangent vectors at $f(\Time)$.
Note that we use an inner product for the orthogonality condition, which is weighted by the approximate solution.
This is indeed the formulation based on the Fisher--Rao information metric for probability density functions, when considering the Fokker--Planck equation as gradient flow, cf. \cite[Sec.~5.2.2]{zhang2024Nonlinpar} and \cite{Anderson2024,Chen2024}.

\subsection{Characterization and properties of the variational approximation}

We choose the admissible parameter fields as
\[
\mathcal Z_\Omega =
\left\{
z\in H^s(\Omega;\mathbb R^m):
z(\overline\Omega)\subset K
\right\},
\]
where \(s>\dim(\Omega)/2\) and \(K\Subset\mathcal Z\). The Sobolev embedding then makes the point-wise positivity constraint of the parametrization $\Phi$ meaningful. We assume that \(z\mapsto\Phi(z,\cdot)\) is sufficiently smooth and that its derivatives satisfy the integrability bounds required for the induced Nemytskii map
\(\Phi(z)(\Position,\Configuration)=\Phi(z(\Position),\Configuration)\)
to be smooth from \(\mathcal Z_\Omega\) into \(\mathcal H\). We further assume that the Fisher information matrix
\[
G(z)
=
\int_\ConfigurationDomain
\frac{
\partial_z\Phi(z,\Configuration)\otimes
\partial_z\Phi(z,\Configuration)
}{
\Phi(z,\Configuration)
}\Diff[\Configuration]
\]
is symmetric, positive definite for every parameter \(z\in\mathcal Z\), locally uniformly on compact subsets of \(\mathcal Z\). Under these assumptions, the parametrization has injective differential and the pulled-back Fisher--Rao tensor defines a nondegenerate Riemannian metric on the ansatz manifold $\Manifold$. This viewpoint is consistent with the recent treatment of parametrized families of Gaussian measures in \cite{LieroMielkeTseZhu_2026}.

For any $f = \Phi(z(\cdot),\cdot)\in\Manifold$ the tangent space writes 
\[
\TangentSpace[f] = \left\{\PDiff{\Parameter} \Phi(\Parameter(\cdot),\cdot)w\mid w\in\R^M\right\}.
\]
We consider a basis $\varphi_1(\Parameter(\Position),\cdot),\ldots,
\varphi_M(\Parameter(\Position),\cdot)$ of the range of $\PDiff{\Parameter}\Phi(\Parameter(\Position),\cdot)$, that is orthonormal in the sense that
\[
\int_\ConfigurationDomain \varphi_m(\Parameter(\Position),\Configuration)\,
\varphi_l(\Parameter(\Position),\Configuration)\,
\frac{\Diff[\Configuration]}{f(\Position,\Configuration)} = \delta_{m,l}       
\]
for all $m,l=1,\ldots,M$. With such a basis we define the partial orthogonal projection 
\begin{align*}
P_f: \Hilbert\to\TangentSpace[f],\quad
P_f \varphi(\Position,\Configuration) = 
        \sum_{m=1}^M \int_\ConfigurationDomain \varphi_m(\Parameter(\Position),\Configuration') \varphi(\Position,\Configuration')\frac{\Diff{\Configuration'}}{f(\Position,\Configuration')}\  \varphi_m(\Parameter(\Position),\Configuration).
\end{align*}
It will allow for an elementary characterization of the variational principle and an abstract a posteriori error representation. 

\begin{proposition}[Variational approximation]
    If the ansatz set $\Manifold$ is a manifold, then the following holds: 
    \begin{enumerate}[label=\normalfont(\roman*)]
        \item The variational orthogonality condition \eqref{eq:var} is  equivalent to the partial linear least squares problem: 
        \[
        \int_{\ConfigurationDomain} 
    \left( \PDiff{\Time}f(\Time,\Position,\Configuration) - \mathcal{L}(\Time) f(\Time,\Position,\Configuration)\right)^2
        \frac{\Diff[\Configuration]}{f(\Time,\Position,\Configuration)} = \min_{\partial_t f(t)} !
        \]
        \item The variational orthogonality condition \eqref{eq:var} is  equivalent to the partially projected evolution equation:
        \[    
    \PDiff{\Time} f(\Time) = P_{f(\Time)}\mathcal L f(\Time).
        \]   
    \item If $\int_\ConfigurationDomain f(0,\Position,\Configuration) \Diff[\Configuration] = 1$ everywhere in $\Omega$ and $f(\Time)\in \TangentSpace[f(\Time)]$, then
        \[
            \int_{\ConfigurationDomain}f(\Time,\Position,\Configuration) \,\Diff[\Configuration] = 1
        \]
        for all times $\Time$, for which the variational approximation exists.
    \item Let $\mathcal{T}(\Time,s)$ with $\Time,s\in [0,T]$ denote the evolution operator associated with the evolutionary equation \eqref{eq:evolution}.
        If $\FPpdf(0) = f(0)$, then
        \[
            \FPpdf(\Time) - f(\Time) = \int_0^t \mathcal{T}(t,s) \ (\Id_\Hilbert - P_{f(s)}) \mathcal L(s) f(s) \,\Diff[s]
        \]
        for all times $\Time$, for which the variational approximation exists.  
    \end{enumerate}
\end{proposition}

\begin{proof}
The statements (i) and (ii) are immediate. As for propagation of normalization (iii), we use \eqref{eq:var} with $\varphi = f(\Time)$ and obtain
    \begin{align*}
        \frac{\Diff}{\Diff[\Time]} \int_{\ConfigurationDomain} f(\Time,\Position,\Configuration) \,\Diff[\Configuration]
            &= \int_{\ConfigurationDomain} \PDiff{\Time} f(\Time,\Position,\Configuration) \, \frac{f(\Time,\Position,\Configuration)}{f(\Time,\Position,\Configuration)} \,\Diff[\Configuration]\\
            &\stackrel{\mathclap{\eqref{eq:var}}}{=} \int_{\ConfigurationDomain} \mathcal L(\Time) f(\Time,\Position,\Configuration)\, \frac{f(\Time,\Position,\Configuration)}{f(\Time,\Position,\Configuration)} \,\Diff[\Configuration]\\
            &= \int_{\ConfigurationDomain} \mathcal L(\Time) f(\Time,\Position,\Configuration) \,\Diff[\Configuration] \stackrel{\eqref{eq:mass}}{=} -\mathcal L_{\Position}(\Time)\int_\ConfigurationDomain f(\Time,\Position,\Configuration)\Diff[\Configuration].
    \end{align*}
    Hence, $m(\Time,\Position) := \int_{\ConfigurationDomain} f(\Time,\Position,\Configuration) \,\Diff[\Configuration]$ satisfies 
    \[
        \partial_\Time m(\Time) = -\mathcal L_{\Position}(\Time) m(\Time),\qquad m(0) = 1.
    \]
    Since $\mathcal L_{\Position}(\Time)1 = 0$, the solution is stationary with $m(t) = 1$.
    For the a posteriori error formula (iv), we calculate the time-derivative of the error,
    \begin{align*}
        \PDiff{\Time} \left( \FPpdf(\Time) - f(\Time) \right)
            &= \mathcal L(\Time) \FPpdf(\Time) - P_{f(\Time)}\mathcal L f(\Time)\\
            &= \mathcal L(\Time) \left( \FPpdf(\Time) - f(\Time) \right) + (\Id_\Hilbert- P_{f(\Time)})\mathcal L(\Time) f(\Time),
    \end{align*}
    and use Duhamel's principle.
\end{proof}

The time-dependent Dirac--Frenkel principle itself does not provide a general positivity-preservation theorem. In the present work, the parametrization $\Phi$ is chosen so that every element in $\Manifold$ is point-wise strictly positive. Hence, the approximate solution $f(t)$ remains point-wise positive for all times on which the reduced dynamics exists.

\section{Gaussian manifold on unbounded domains}
\label{sec:hookean}

In this section, we show that the macroscopic closure for the extra-stress given by \eqref{eq:extrai}
can be also obtained by a variational principle. To be more precise, the
variational formulation is based on the Dirac–Frankel principle applied to
a Gaussian approximation manifold.  This manifold is invariant under the linear
configurational dynamics, and we recover \eqref{eq:extrai}.
Hence, these different approaches yield an equivalent macroscopic formulation.

Recall that we consider the linear Hookean force model for the FP-NS system, such that $\ConfigurationDomain=\R^{\dm}$ with $\dm = Nd$. In this section, we assume that the physical domain $\Omega$ and the initial conditions for $\Velocity$ and $\Covariance$ are regular enough.
Motivated by the Gaussian stationary solution \eqref{eq:stationary}, we explore an approximation of the form 
\begin{align}
    \label{eq:ansatz}
    \FPpdf(\Time, \Position, \Configuration) \approx 
    \prod_{n=1}^N \frac{\exp\left(-\frac{1}{2} \Configuration_n\Transpose \Conformation_n(\Time,\Position)^{-1} \Configuration_n\right)}{\sqrt{(2\pi)^d \det(\Conformation_n(\Time, \Position))}}
    = \frac{\exp\left(-\frac{1}{2} \Configuration\Transpose \Covariance(\Time,\Position)^{-1} \Configuration\right)}{\sqrt{(2\pi)^{\dm} \det(\Covariance(\Time, \Position))}},
\end{align}
where $\Conformation_n(\Time, \Position) \in \Matsym[d]$ are symmetric and positive definite covariance matrices, and $\Covariance := \mathrm{diag}(\Conformation_1,\dots,\Conformation_N) \in \Matsym$ is the global block-diagonal covariance matrix.
We thus work in the variational setting choosing a nonlinear parametrization
\[
\Phi:\ParameterDomain\times\ConfigurationDomain\to (0,\infty),\quad
(\Covariance,\Configuration)\mapsto 
 \frac{\exp\left(-\frac{1}{2} \Configuration\Transpose \Covariance^{-1} \Configuration\right)}
{\sqrt{(2\pi)^{\dm} \det(\Covariance)}},
\]
that maps the smooth manifold $\ParameterDomain$ of block-diagonal symmetric and positive definite matrices with $N$ blocks of size $d\times d$, which is of dimension $M= N d(d+1)/2$, to a normalized Gaussian.
Here, the corresponding approximation manifold reads
\begin{align*}
    \Manifold = \Set*{ f\in\Hilbert \given f(\Position,\Configuration) = \frac{\exp\left(-\frac{1}{2} \Configuration\Transpose \Covariance(\Position)^{-1} \Configuration\right)}{\sqrt{(2\pi)^{Nd} \det(\Covariance(\Position))}} ,\ \Covariance : \Domain\to\ParameterDomain }.
\end{align*}
These Gaussian probability distributions have their covariance matrices as their second moments and thus provide simple access to the extra-stress tensor.
Indeed, we have for any $f\in\Manifold$ with $f(\Position,\Configuration)=\Phi(\Covariance(\Position),\Configuration)$ that
\begin{align}
    \ExtraStress[f](\Position) 
    = \int_{\ConfigurationDomain} f(\Position, \Configuration) \sum_{n=1}^N (\Configuration_n\Configuration_n\Transpose - \Id_d) \,\Diff[\Configuration]
    = \sum_{n=1}^N (\Conformation_n(\Position) - \Id_d). \label{eq:ansatz:tau}
\end{align}
In particular, the diagonal blocks $\Conformation_n(\Position)$ of the Gaussian covariance matrice are the summands of the conformation tensor.

\subsection{Derivation of the Hookean approximate solution}

In the following, we construct the variational approximation $f(t)\approx\FPpdf(t)$ given by the orthogonality principle \eqref{eq:var} proposed in \cref{sec:ansatz}.
We start by having a closer look at the tangent spaces of the Gaussian manifold and their mapping properties with respect to the differential operators of the Fokker--Planck equation.
The key finding is, that the linear Hookean force $\Force(\Configuration) = \Configuration$, $\Configuration\in\ConfigurationDomain$, defines a configurational differential operator~\eqref{eq:Lc_Hook}, which is exactly represented in the Gaussian tangent spaces. 

\begin{lemma}[Tangent space and differential expressions]
    \label{lem:tangent-space}
    Consider a Gaussian probability distribution $f\in \Manifold$ with $f(\Position,\Configuration) = \Phi(\Covariance(\Position),\Configuration)$ and $\Covariance(\Position)\in\ParameterDomain$.
    \begin{enumerate}[label=\normalfont(\roman*)]
    \item The tangent space of $\Manifold$ at $f$ satisfies
        \begin{align*}
            \TangentSpace[f] = \Set*{ \varphi\in\Hilbert \given \varphi(\Position,\Configuration) = \left(\Configuration \Transpose \tensor{A}(\Position) \Configuration - \tr(\tensor{A}(\Position) \Covariance(\Position))\right)f(\Position,\Configuration) ,\ 
            \tensor{A} : \Domain\to \Tilde{\mathcal{Z}} }.
        \end{align*}
        where $\Tilde{\mathcal{Z}} \subset \Matsym$ denotes the space of block-diagonal symmetric matrices consisting of $N$ blocks of size $d\times d$.
    \item The configurational differential operator satisfies $\mathcal{L}_{\Configuration}(\Time)f\in\TangentSpace[f]$.
        In particular, 
        \begin{align}
            \mathcal{L}_{\Configuration}(\Time) f
            &= \left(\Configuration\Transpose \tensor{A}_{\Configuration}(\Time) \Configuration
                - \tr\left(\tensor{A}_{\Configuration}(\Time)\Conformation \right) \right)f , \label{eq:configuration}
        \end{align}
        where $\tensor{A}_{\Configuration}(\Time,\Position) \in \Tilde{\mathcal{Z}}$ is given by
        \begin{align}
            \tensor{A}_{\Configuration}(\Time,\Position) =
            \frac12\Covariance(\Position)^{-1} \left(\tensor{M}(\Time,\Position)\Covariance(\Position) + \Covariance(\Position)\tensor{M}(\Time,\Position)\Transpose - \frac{1}{\Deborah} \tensor{\Lambda} \otimes \Id_d \right) \Covariance(\Position)^{-1}.
            \label{eq:configuration:tensors}
        \end{align}
    \item The spatial differential operator 
        $\mathcal{L}_{\Position}(\Time) = \Velocity(\Time)\cdot\GradX-\CenterDiffusion\Laplace_{\Position}$ satisfies
        \begin{equation}\label{eq:space}
            \mathcal{L}_{\Position}(\Time) f 
            = \left(\Configuration\Transpose \tensor{A}_{\Position}(\Time) \Configuration - \tr(\tensor{A}_{\Position}(\Time)\Covariance) \right) f + \CenterDiffusion \rho f 
        \end{equation}
        where $\tensor{A}_{\Position}(\Time,\Position) \in \Tilde{\mathcal{Z}}$ and the scalar remainder $\rho(\Position, \Configuration)\in\R$ are given by
        \begin{align*}
        \tensor{A}_{\Position}(\Time,\Position)
            &= \frac12 \Covariance(\Position)^{-1} \big(\mathcal{L}_{\Position}(\Time)\Covariance(\Position)\big) \Covariance(\Position)^{-1}\\ 
        \rho(\Position,\Configuration)
            &= \sum_{i=1}^d \Configuration\Transpose \big(\Covariance(\Position)^{-1} \PDiff{x_i}\Covariance(\Position)\big)^2 \Covariance(\Position)^{-1} \Configuration
            - \frac{1}{2}\sum_{i=1}^d \tr\left(\left(\Covariance(\Position)^{-1} \PDiff{x_i}\Covariance(\Position)\right)^2\right) \\&\quad
            -\frac14 \sum_{i=1}^d \left[ \Configuration\Transpose \Covariance(\Position)^{-1} \PDiff{x_i}\Covariance(\Position) \Covariance(\Position)^{-1} \Configuration - \tr\left(\Covariance(\Position)^{-1}\PDiff{x_i}\Covariance(\Position)\right)\right]^2
        \end{align*}
        In particular, if $\CenterDiffusion = 0$, then $\mathcal{L}_{\Position}(\Time) f\in\TangentSpace[f]$.
    \end{enumerate}
\end{lemma}

\begin{proof}
    As for (i), we consider a curve $\Covariance_s(\Position)\in\ParameterDomain$ passing through $\Covariance(\Position)$ at $s=0$ and differentiate
    \[
    \frac{\Diff}{\Diff[s]} \det(\Covariance_s(\Position))^{-1/2} 
    = -\frac12 \det(\Covariance_s(\Position))^{-1/2} \tr\left(\Covariance_s(\Position)^{-1}\dot \Covariance_s(\Position)\right),
    \]
    using Jacobi's formula. Therefore, 
    \begin{align*}
        \frac{\Diff}{\Diff[s]} \Phi(\Covariance_s(\Position),\Configuration) =
        \frac12 \Big(\Configuration^\top \Covariance_s(\Position)^{-1}\dot\Covariance_s(\Position) \Covariance_s(\Position)^{-1}\Configuration
        -\tr\big(\Covariance_s(\Position)^{-1}\dot\Covariance_s(\Position)\big) \Big)\Phi(\Covariance_s(\Position),\Configuration),
    \end{align*}
    so that matrix $\tensor{A}(\Position) = \frac12 \Covariance(\Position)^{-1}\dot\Covariance_s(\Position) |_{s=0} \Covariance(\Position)^{-1}$ is block-diagonal and symmetric, and thus establishes our claim. 
    Note that this is a special case of \cite[Lemma 3.1]{Lasser2020}.
    \\
    As for the derivatives in configuration space, that is (ii), we compute
    \begin{gather*}
        \GradQ f = - \Covariance^{-1} \Configuration f,
    \qquad
        \Laplace_{\Configuration} f = \left(\Configuration\Transpose \Covariance^{-2} \Configuration - \tr\!\left(\Covariance^{-1}\right)\right) f .
    \end{gather*}
    Hence, we obtain
    \begin{align*}
        \mathcal L_{\Configuration}(\Time) f 
        &= - \DivQ\left( \tensor{M}(t)\Configuration f + \tfrac{1}{2\Deborah} (\tensor{\Lambda} \otimes \Id_d) \GradQ f\right)\\*[1ex]
        &= \Configuration\Transpose\left(\Covariance^{-1} \tensor{M}(\Time) - \tfrac{1}{2\Deborah} (\tensor{\Lambda} \otimes \Id_d) \Covariance^{-2}\right)\Configuration f
            - \tr(\tensor{M}(\Time) - \tfrac{1}{2\Deborah} (\tensor{\Lambda} \otimes \Id_d) \Covariance^{-1}) f \\*[1ex]
        &= \left( \Configuration\Transpose \tensor{A}_{\Configuration}(\Time)\Configuration -\tr(\tensor{A}_{\Configuration}(\Time)\Covariance) \right) f,
    \end{align*}    
    where we used that $(\tensor{\Lambda} \otimes \Id_d)$ and $\Covariance$ commute, and that $\Configuration\Transpose \tensor{B} \Configuration = \frac{1}{2}\Configuration\Transpose (\tensor{B} + \tensor{B}\Transpose) \Configuration$ for all $\tensor{B} \in \R^{d \times d}$.
    In particular, $\tensor{A}_{\Configuration}$ is block-diagonal and symmetric. \\
    As for (iii), we compute the spatial differentials of the Gaussian function as
    \begin{align*}
        \PDiff{x_i} f
        &= \frac{1}{2} \left(\Configuration\Transpose \Covariance^{-1} \PDiff{x_i}\Covariance \Covariance^{-1} \Configuration
            -\tr(\Covariance^{-1}\PDiff{x_i}\Covariance) \right) f ,
    \end{align*}
    and
    \begin{align*}
        \Laplace_{\Position} f
        &= \frac{1}{2} \sum_{i=1}^d \PDiff{x_i} \left( \Configuration\Transpose \Covariance^{-1} \PDiff{x_i}\Covariance \Covariance^{-1} \Configuration
            - \tr\left(\Covariance^{-1} \PDiff{x_i}\Covariance\right) \right) f \\&\qquad
         + \frac{1}{4} \sum_{i=1}^d \left( \Configuration\Transpose \Covariance^{-1} \PDiff{x_i}\Covariance \Covariance^{-1} \Configuration
            - \tr\left(\Covariance^{-1} \PDiff{x_i}\Covariance\right) \right)^2 f \\*[1ex]
        &= \frac12 \left(\Configuration\Transpose \Covariance^{-1} \Laplace_{\Position}\Covariance \Covariance^{-1} \Configuration
            - \tr\left(\Covariance^{-1} \Laplace_{\Position}\Covariance\right) \right)f\\&\qquad
           - \sum_{i=1}^d \left( \Configuration\Transpose \left(\Covariance^{-1} \PDiff{x_i}\Covariance\right)^2 \Covariance^{-1} \Configuration
            - \frac12\tr\left(\left(\Covariance^{-1} \PDiff{x_i}\Covariance\right)^2\right) \right) f\\&\qquad
             + \frac{1}{4} \sum_{i=1}^d \left( \Configuration\Transpose \Covariance^{-1} \PDiff{x_i}\Covariance \Covariance^{-1} \Configuration - \tr\left(\Covariance^{-1} \PDiff{x_i}\Covariance\right) \right)^2 f.
    \end{align*}
    Combining these, we obtain
    \begin{align*}
        \mathcal{L}_{\Position}(\Time) f 
        &= (\Velocity(\Time) \cdot \GradX - \CenterDiffusion \Laplace_{\Position})f\\*[1ex]
        &= \frac{1}{2} \left(\Configuration\Transpose \Covariance^{-1}  (\mathcal{L}_{\Position}(\Time)\Covariance) \Covariance^{-1} \Configuration - \tr\left(\Covariance^{-1}(\mathcal{L}_{\Position}(\Time)\Covariance)\right) \right) f\\&\quad
            + \CenterDiffusion\sum_{i=1}^d \left( \Configuration\Transpose \left(\Covariance^{-1} \PDiff{x_i}\Covariance\right)^2 \Covariance^{-1} \Configuration
            - \frac12\tr\left(\left(\Covariance^{-1} \PDiff{x_i}\Covariance\right)^2\right) \right) f\\&\quad
         - \frac{\CenterDiffusion}{4} \sum_{i=1}^d \left( \Configuration\Transpose \Covariance^{-1} \PDiff{x_i}\Covariance \Covariance^{-1} \Configuration - \tr\left(\Covariance^{-1} \PDiff{x_i}\Covariance\right) \right)^2 f \\*[1ex]
         &= \left(\Configuration\Transpose \tensor{A}_{\Position}(\Time) \Configuration - \tr\left(\Covariance \tensor{A}_{\Position}(\Time)\right) \right) f
            + \CenterDiffusion \rho(\Position, \Configuration) f
    \end{align*}
    with the claimed $\tensor{A}_{\Position}(\Time,\Position)$ and $\rho(\Position,\Configuration)$.
\end{proof}

The action of the spatial operator $\mathcal{L}_{\Position}(\Time)$ on a Gaussian $f$ creates a term caused by the center diffusion, that completely lies in the orthogonal complement of the tangent space of $\Manifold$ at $f$.

\begin{proposition}[Orthogonality of spatial remainder] \label{prop:orthogonality}
    For a Gaussian probability distribution $f\in \Manifold$, the scalar remainder $\rho$ of the spatial representation \eqref{eq:space} satisfies $\rho f \perp \TangentSpace[f]$ w.r.t $P_f$.
\end{proposition}

\begin{proof}
    We show the equivalent statement that $\rho \perp \TangentSpace[f]$ in $L^2$.
    To this end, we rewrite $\rho$ using the substitutions $\vec{r} = \Covariance^{-1/2}\Configuration$ and $\tensor{D}_i = \Covariance^{-1/2} \PDiff{x_i}\Covariance \Covariance^{-1/2}$ to obtain 
    \begin{align}
        \label{eq:Rterm}
        \rho(\Covariance^{1/2}\vec{r})
        &= \sum_{i=1}^d \underbrace{\Big( \vec{r}\Transpose \tensor{D}_i^2 \vec{r}
            - \tfrac{1}{2} \tr\left(\tensor{D}_i^2\right)
            - \tfrac{1}{4} \left[ \vec{r}\Transpose \tensor{D}_i \vec{r} - \tr(\tensor{D}_i) \right]^2 \Big)}_{=:\,\rho_i(\vec{r})}.
    \end{align}
    As $\tensor{D}_i$ is symmetric it can be written as $\tensor{D}_i = {\tensor{Q}}_i\Transpose \tensor{\Lambda}_i \tensor{Q}_i$ with $\tensor{\Lambda}_i$ being a diagonal matrix with eigenvalues
    $\lambda_{i,k}$, $ k=1, \ldots, \dm$, and $\tensor{Q}_i$ being orthogonal.
    Then each $\vec{r} \in \R^\dm$ can be obviously written as linear combination of the associated eigenvectors $\vec{v}_{i,k}$, i.e., $ \vec{r} = \sum_{k=1}^\dm \alpha_{i,k} \vec{v}_{i,k}$.
    Reformulating each summand term, $i = 1,\dots,\dm$, on the right of \eqref{eq:Rterm} yields
    \begin{align*}
        \rho_i(\vec{r}) 
        &= \vec{r}\Transpose \tensor{D}_i^2 \vec{r}
            - \frac{1}{2} \tr\left(\tensor{D}_i^2\right)
            - \frac{1}{4} \left[ \vec{r}\Transpose \tensor{D}_i \vec{r} - \tr(\tensor{D}_i) \right]^2 \\
        &= \sum_{k=1}^\dm \lambda_{i,k}^2 \left( \alpha_{i,k}^2 - \tfrac{1}{2} \right)
            - \frac{1}{4} \left[ \sum_{k=1}^\dm \lambda_{i,k} \left( \alpha_{i,k}^2 - 1 \right) \right]^2 \\
        &= -\frac{1}{4} \sum_{k=1}^\dm \lambda_{i,k}^2 \left(\alpha_{i,k}^4 - 6 \alpha_{i,k}^2 + 3\right)
            - \frac{1}{2} \sum_{k=1}^\dm \sum_{n=1}^{k-1} \lambda_{i,k} \lambda_{i,n} \left( \alpha_{i,k}^2 - 1 \right) \left( \alpha_{i,n}^2 - 1 \right) \\
        &= -\frac{1}{4} \sum_{k=1}^\dm \lambda_{i,k}^2 H_4(\alpha_{i,k})
            - \frac{1}{2} \sum_{k=1}^\dm \sum_{n=1}^{k-1} \lambda_{i,k} \lambda_{i,n} H_2(\alpha_{i,k}) H_2(\alpha_{i,n}) ,
    \end{align*}
    where $H_2$ and $H_4$ are the second and fourth one-dimensional Hermite polynomial, respectively. 
    Since $\TangentSpace[f] \subset f P_2$, where $P_2$ denotes polynomials (in~$\Configuration$) of total degree $\le 2$, we can write any function $\varphi \in \TangentSpace[f]$ as
    \begin{align*}
        \varphi(\Configuration) &= \varphi(\Covariance^{1/2}\vec{r}) = \frac{\exp\left(-\frac{|\vec{r}|^2}{2}\right)}{\sqrt{(2\pi)^\dm \det(\Covariance)}} g(\vec{r}) \\
        g(\vec{r}) &= \sum_{\genfrac{}{}{0pt}{2}{0 \leq m_1, \ldots , m_\dm \leq 2}{\sum_{k=1}^\dm m_k \leq 2}} \beta_{i,m_1, \ldots , m_\dm} \prod_{j=1}^\dm  H_{m_j}(\alpha_{i,j}).
    \end{align*}
    Hence, testing $\rho$ with $\varphi \in \TangentSpace[f]$ yields
    \begin{align*}
        \int_{\ConfigurationDomain} \rho(\Configuration) \varphi(\Configuration) \,\Diff[\Configuration]
        &= \int_{\ConfigurationDomain} \rho(\Covariance^{1/2}\vec{r}) \varphi(\Covariance^{1/2}\vec{r}) \det(\Covariance^{1/2})\,\Diff[\vec{r}] \\
        &= \sum_{i=1}^d \underbrace{\int_{\ConfigurationDomain} \frac{\exp\left(-\frac{|\vec{r}|^2}{2}\right)}{\sqrt{(2\pi)^\dm}} \rho_i(\vec{r}) g(\vec{r}) \,\Diff[\vec{r}]}_{=:\, I_i} .
    \end{align*}
    In the following, we show that all $I_i$ are zero.
    Setting $\vec{\alpha}_i := (\alpha_{i,1} , \ldots , \alpha_{i,\dm})\Transpose$, we get $|\vec{r}| = |\vec{\alpha}_i|$ and $\int_{\ConfigurationDomain} \ldots \Diff[\vec{r}] = \int_{\ConfigurationDomain} \ldots \Diff[\vec{\alpha}_i]$.
    Now the weighted orthogonality of the Hermite polynomials yields that for all indices $i,k \leq \dm$ and $ 0 \leq j \leq 2$, we get
    \begin{align} \label{eq:pol2}
        \int_{\R} \exp\left(-\tfrac{\alpha_{i,k}^2}{2} \right)  H_4(\alpha_{i,k})   H_{j}(\alpha_{i,k}) \,\Diff[\alpha_{i,k}] =0 .
    \end{align}
    As in the representation of $g(\cdot)$ we have $m_n + m_k \leq \sum_{l=1}^\dm m_l \leq 2$, there is either $m_n$ or $m_k$ strictly smaller than two.
    Then the orthogonality yields
    \begin{align} \label{eq:pol4}
        \int_{\R} \int_{\R} \exp\left(-\tfrac{\alpha_{i,k}^2}{2} \right) \exp\left(-\tfrac{\alpha_{i,n}^2}{2} \right)  H_2(\alpha_{i,k})  H_2(\alpha_{i,n})
        H_{m_k}(\alpha_{i,k})  H_{m_n}(\alpha_{i,n})\,\Diff[\alpha_{i,k}]\Diff[\alpha_{i,n}]  =0 .
    \end{align}  
    All summation terms in the definition of $I_i$ have a factor in the form of either \eqref{eq:pol4} or \eqref{eq:pol2} and thus $I_i =0$.
\end{proof}

Combining \cref{lem:tangent-space} and \cref{prop:orthogonality}, we obtain the variational equations of motion for the Hookean Fokker--Planck equations.
We note that the variational equations of motion are precisely the diffusive Oldroyd-B model~\eqref{eq:extrai}.

\begin{proposition}[Equations of motion]
    \label{prop:parameters}
    The solution to the variational approximation by \cref{eq:var} in the Gaussian manifold is given by $f(\Time)\in \Manifold$ with $f(\Time,\Position,\Configuration) = \Phi(\Covariance(\Time,\Position),\Configuration)$ and $\Covariance(\Time,\Position)\in\ParameterDomain$ if and only if the conformation tensor $\Covariance(\Time, \Position)$ satisfies 
    \begin{equation}\label{eq:eom}
    \PDiff{\Time} \Covariance + (\Velocity \cdot \GradX) \Covariance - \CenterDiffusion \Laplace_{\Position} \Covariance
            = \frac{1}{\Deborah} (\tensor{\Lambda} \otimes \Id_d) - \Covariance \tensor{M}\Transpose - \tensor{M} \Covariance.
    \end{equation}
    In particular, the variational approximation of the extra-stress tensor \eqref{eq:extraStress} coincides with the exact one.
    The variational approximation even exactly solves the Hookean Fokker--Planck equation \eqref{eq:FP} when $\CenterDiffusion = 0$, i.e., without center-of-mass diffusion.
\end{proposition}

\begin{proof}
    This follows directly from combining the differential operators in space and configuration using \cref{lem:tangent-space} and the orthogonality result of \cref{prop:orthogonality}.
\end{proof}

\subsection{Existence of unique solutions}

In the following, we study the well-posedness of the equation of motion for the covariance matrix $\Covariance(\Time,\Position)$ of the variational Gaussian approximation.
To simplify the exposition, we only consider \cref{eq:eom} for an arbitrary, regular enough, divergence-free, given velocity field $\Velocity(\Time,\Position)$.
For the case of full coupling with the Navier--Stokes equations, we refer to \cite{Barrett2018,Debiec2025}.
Note that therein the authors need to introduce a defect measure in the momentum equation to account for the lack of compactness in the extra-stress tensor.

\begin{lemma}
    \label{lem:existence}
    Let $\Velocity \in L^\infty(0, T; W^{1,\infty}(\Domain))$ and $\tensor{M} = \frac{1}{2\Deborah} (\tensor{\Lambda} \otimes \Id_d) - \GradX\Velocity$ be given.
    Then, the weak solution $\Covariance$ of the advection-diffusion-reaction equation
    \begin{gather*}
        \PDiff{\Time} \Covariance + (\Velocity \cdot \GradX) \Covariance - \CenterDiffusion \Laplace_{\Position} \Covariance
            = \frac{1}{\Deborah} (\tensor{\Lambda} \otimes \Id_d) - \Covariance \tensor{M}\Transpose - \tensor{M} \Covariance ,\qquad\text{in }(0,T) \times \Domain \\
        \GradX\Covariance \cdot \vec{n} |_{(0,T) \times \partial\Domain}  = \tensor{0}, \qquad
        \Covariance|_{\Time = 0} = \Covariance_0 ,
    \end{gather*}
    takes its values in the set of symmetric positive definite matrices.
\end{lemma}

\begin{proof}
    We note that the differential operators of the advection-diffusion-reaction equation act element-wise on $\Covariance(\Time)$, while its right hand side matrix part satisfies $(\Covariance(\Time) \tensor{M}(\Time)\Transpose + \tensor{M}(\Time) \Covariance(\Time))\Transpose = \Covariance(\Time)\Transpose \tensor{M}(\Time)\Transpose + \tensor{M}(\Time) \Covariance(\Time)\Transpose$. Hence, 
    $\PDiff{\Time}(\Covariance(\Time)\Transpose-\Covariance(\Time)) = 0$,
    so that symmetry is preserved.
    To verify positive definiteness, we use the evolution operator $\Flow(\Time, \Time_0)$
    of the linear operator $\tensor{M}_\partial(\Time)$,
    \[
        \tensor{M}_{\partial}(\Time) = -\tensor{M}(\Time) - \frac12  (\Velocity \cdot \GradX)
        + \frac12 \CenterDiffusion \Laplace_{\Position}.
    \]
    The elementwise Neumann Laplacian $A = \frac12 \CenterDiffusion \Laplace_{\Position}$ is a self-adjoint linear operator on the domain $\D = \{\varphi\in H^1(\Omega; \R^{\dm\times\dm}), \ \Delta\varphi\in L^2(\Omega; \R^{\dm\times\dm}),\  \GradX\varphi \cdot \vec{n} = 0\ \text{weakly}\}$, see \cite[\S5.3.3]{Arendt2002}.
    The spectrum is discrete, contained in the left half-axis $]-\infty,0]$, with $0$ as an eigenvalue, so that $A$ generates a $C_0$ semigroup $\exp(As)$ on $L^2(\Omega; \R^{\dm\times\dm})$, which is contractive in the sense that $\|\exp(As)\|\le 1$ for all $s\ge 0$.
    We also observe that for fixed $\Time$, the operator $B(\Time) = -\frac12 (\Velocity(\Time) \cdot \GradX)$ is relatively bounded with respect to the Neumann Laplacian $A$.
    Indeed, by Sobolev embedding \cite[Theorem~7.28]{Gilbarg2001}, for any $a>0$ there exists $b>0$ such that for all $\varphi\in \D$
    \begin{align*}
        \int_\Omega |\Velocity(\Time,\Position) \cdot \GradX\varphi(\Position)|^2 \,\Diff\Position
            &\leq \|\Velocity(\Time,\Position)\|^2_{L^\infty(\Omega)} \int_\Omega |\GradX\varphi(\Position)|^2 \,\Diff\Position \\
            &\leq \|\Velocity(\Time,\Position)\|^2_{L^\infty(\Omega)} \left(a \int_\Omega |\Laplace_{\Position} \varphi(\Position)|^2 \,\Diff\Position + b \int_\Omega|\varphi(\Position)|^2 \,\Diff\Position\right).
    \end{align*}
    Moreover,
    \begin{align*}
        \int_\Omega \varphi(\Position)\, \Velocity(\Time,\Position)\cdot\GradX\varphi(\Position) \,\Diff\Position
        = \int_\Omega \tfrac{1}{2} \DivX \big(\varphi^2(\Position) \Velocity(\Time,\Position)\big) \,\Diff\Position
        = \int_{\partial\Omega} \tfrac{1}{2} \varphi^2 \Velocity(\Time,\Position) \cdot \vec{n} \,\Diff[s] = 0
    \end{align*}
    for all $\varphi\in \D$, so that $B(\Time)$ is dissipative. 
    Then, by unbounded perturbation theory, see \cite[III.2]{Engel2000}, $A+ B(\Time)$ generates a $C_0$ semigroup of contractions $S_{\Time}(s)$, $s\in[0,T]$, on $L^2(\Omega; \R^{\dm\times\dm})$, which due to its contractivity is stable, see \cite[\S5.2]{Pazy1983}.
    Moreover, there exist a constant $K>0$, depending on a uniform bound on $\GradX\Velocity(\Time,\Position)$ and the Deborah number, such that $\|\tensor{M}(\Time)\|\le K$ for all $\Time\in[0,T]$.
    Hence, $\tensor{M}_\partial(\Time) = A + B(\Time) - \tensor{M}(\Time)$ is a stable family of generators with time independent domain $\D$, see \cite[\S5.2]{Pazy1983}. 
    By \cite[\S5.4, Theorem~4.8]{Pazy1983}, we then have a well-defined evolution operator $\Flow(\Time,\Time_0)$. 
    It satisfies for all $\Covariance_0\in \D$
    \[
        \PDiff{\Time}\Flow(\Time,\Time_0)\Covariance_0 = 
        \tensor{M}_\partial(\Time)\Flow(\Time,\Time_0)\Covariance_0,\qquad \Flow(\Time_0,\Time_0)\Covariance_0 = \Covariance_0.
    \]
    We also consider the transpose operator $\tensor{M}_\partial(\Time)\Transpose = A + B(\Time) - \tensor{M}(\Time)\Transpose$, which acts with the differential operators $A$ and $B(t)$ on any matrix entry and multiplies with the matrix $\tensor{M}(\Time)\Transpose$ from the right.
    The evolution operator $\Flow(\Time,\Time_0)\Transpose$ of $\tensor{M}_\partial(\Time)\Transpose$ is the transpose of $\Flow(\Time,\Time_0)$ in the sense that $(\Flow(\Time,\Time_0)\Covariance_0)\Transpose = \Covariance_0\Flow(\Time,\Time_0)\Transpose$ for all $\Covariance_0\in \D$.
    Moreover, 
    \[
        \PDiff{\Time}\Covariance_0 \Flow(\Time,\Time_0)\Transpose = \Covariance_0\Flow(\Time,\Time_0)\Transpose\tensor{M}_\partial(\Time)\Transpose,\qquad \Covariance_0\Flow(\Time_0,\Time_0)\Transpose = \Covariance_0.
    \]
    The two evolutions allow to represent the advection-diffusion-reaction solution $\Covariance(t)$ in a quadratic way as $\Covariance(\Time) = \widetilde\Covariance(\Time)$ with
    \[
        \widetilde\Covariance(\Time) = \Flow(\Time,0) \Covariance_0\Flow(\Time,0)\Transpose 
            + \frac{1}{\Deborah} \int_0^\Time \Flow(\Time,s) (\tensor{\Lambda} \otimes \Id_d) \Flow(\Time,s)\Transpose \,\Diff[s].
    \]
    Indeed, $\widetilde\Covariance(0) = \Covariance_0$ and
    \begin{align*}
        \PDiff{\Time}\widetilde\Covariance(\Time) &= 
        \tensor{M}_\partial(\Time)\Flow(\Time,0) \Covariance_0 \Flow(\Time,0)\Transpose
         + \Flow(\Time,0) \Covariance_0 \Flow(\Time,0)\Transpose\tensor{M}_\partial(\Time)\Transpose + 
         \frac{1}{\Deborah} (\tensor{\Lambda} \otimes \Id_d) \\
         &\quad+ \frac{1}{\Deborah} \int_0^t \Big(\tensor{M}_\partial(\Time)\Flow(\Time,s) (\tensor{\Lambda} \otimes \Id_d)\Flow(\Time,s)\Transpose
         + \Flow(\Time,s) (\tensor{\Lambda} \otimes \Id_d) \Flow(\Time,s)\Transpose \tensor{M}_\partial(\Time)\Transpose \Big)\Diff[s]\\
         &= \tensor{M}_\partial(\Time) \widetilde\Covariance(\Time) + 
         \widetilde\Covariance(\Time)\tensor{M}_\partial(\Time)\Transpose 
         + \frac{1}{\Deborah}(\tensor{\Lambda} \otimes \Id_d).
    \end{align*}
    Using the representation for the quadratic form associated with the matrix $\Covariance(\Time)$, we obtain for any $\vec{v} \in \R^{\dm}$, $\vec{v}\neq0$,
    \begin{align*}
        \vec{v}\Transpose \Covariance(\Time) \vec{v}
        = \vec{v}\Transpose \Flow(\Time,0) \Covariance_0 \Flow(\Time,0)\Transpose \vec{v}
            + \frac{1}{\Deborah} \int_0^\Time \vec{v}\Transpose \Flow(\Time,s) (\tensor{\Lambda} \otimes \Id_d) \Flow(\Time,s)\Transpose \vec{v} \,\Diff[s]
        > 0,
    \end{align*}
    since the evolution operators are invertible for $\Time>0$. 
    Thus, $\Covariance(\Time)$ takes values in the set of symmetric, positive definite matrices.
\end{proof}

\section{Conclusion and future work}
\label{sec:conclusion}
In this work, we have shown equivalence between the classical moment closure of the Hookean linear-chain model and a nonlinear variational approximation based on the Dirac--Frenkel principle on a Gaussian manifold.
This equivalence results from the fact that the Gaussian manifold is invariant under the configurational Fokker--Planck operator associated with linear spring forces.
Thus the variationally projected dynamics coincides with the exact macroscopic closure.
While this equivalence is worked out in full detail for the Hookean linear-chain model, general chains but linear spring forces can be placed into the same abstract framework and involve then the incidence matrix of the associated connected graph.
In the case of FENE models and, more generally, for polymer models with nonlinear forcing laws, this invariance property does not hold any more.
Gaussian probability densities are no longer preserved by the configurational dynamics, and it is well known that exact macroscopic closure relations are not available.
Nevertheless, the variational framework developed in this paper remains applicable and provides a systematic mechanism for constructing reduced models and algorithmic approximations.
In particular, the projected evolution equation gives a closed reduced dynamics on a specific nonlinear approximation manifold.

From an algorithmic point of view, this observation provides an abstract framework to reduced numerical simulation schemes of FENE-type models.
Rather than discretizing the high-dimensional configuration space, e.g., by a spectral method, one can approximate the probability density by a parametrized family of distributions and set the corresponding parameters dynamically.
The reduced evolution equations are then obtained directly from the variational principle.
This pathline guarantees a variational consistent structure allowing for a rich family of alternative methods in situations where no exact closure exists.
Moreover, the derived a posteriori error representation offers a natural tool for assessing the quality of such reduced approximations.
In the non-Hookean setting, the residual term quantifies the modeling error.
This information can be exploited algorithmically, for instance to compare different parametrizations or to set up adaptive strategies monitoring nonlinear effects.
Finally, it is of interest to study the interaction of variationally reduced configurational dynamics with the full micro–macro coupling.
The development and analysis of stable numerical schemes based on nonlinear variational approximations, as well as their long-time behavior for nonlinear polymer models will be investigated in future research work.

\section*{CRediT author statement}

\textbf{C. Lasser:}
    Conceptualization,
    Formal analysis,
    Methodology,
    Writing – original draft,
    Writing – review and editing.

\textbf{S. B. Lunowa:}
    Conceptualization,
    Formal analysis,
    Methodology,
    Writing – original draft,
    Writing – review and editing.

\textbf{B. Wohlmuth:}
    Conceptualization,
    Formal analysis,
    Methodology,
    Writing – original draft,
    Writing – review and editing.

\bibliographystyle{elsarticle-num}
\bibliography{references}

\end{document}